\documentclass{amsart}
\usepackage{amsmath, amssymb, amsthm, enumitem}

\newtheorem{theorem}{Theorem}[section]

\newtheorem{corollary}{Corollary}[section]

\theoremstyle{definition}

\title[Assertibility, truth, and meaningfulness]{A formal system for reasoning about assertibility, truth, and meaningfulness}
\author{Nik Weaver}
\date{October 2025}

\begin{document}

\begin{abstract}
We propose axioms governing the interaction of constructive assertibility and meaningfulness predicates with a self-applicative truth predicate characterized by the T-scheme, and we prove the consistency of the resulting formal system.
\end{abstract}

\maketitle

In \cite{weaver} several formal systems were presented for reasoning in various contexts about constructive truth or {\it assertibility}. Each of these systems uses intuitionistic logic and has an assertibility predicate which applies, via G\"odel numbering or some other form of syntactic encoding, to the system's own sentences. That includes sentences which contain the assertibility predicate itself, i.e., this predicate is self-applicative.

Paradoxes are avoided in these systems because assertibility is not assumed to satisfy the T-scheme $\phi \leftrightarrow \mathbb{T}[\phi]$ which characterizes classical truth. We do adopt the {\it capture} law $$\phi \to \mathbb{A}[\phi]$$ ($\ulcorner$if $\phi$, then ``$\phi$'' is assertible$\urcorner$) for all sentences $\phi$, but the reverse {\it release} law which infers $\phi$ from $A[\phi]$ only appears as a deduction rule, not an implication. In the presence of intuitionistic logic, this is enough to avoid contradiction. All of the formal systems in \cite{weaver} are provably consistent.

The goal of the present paper is to show that assertibility and meaningfulness predicates, both governed by intuitionistic logic, can be used in conjunction to reason consistently about a self-applicative classical truth predicate that verifies the T-scheme. The idea that the global concept of classical truth must be treated constructively should not be too surprising. For any well-defined collection of sentences $\mathcal{S}$ and any partial truth predicate $\mathbb{T}_\mathcal{S}$ that verifies the T-scheme for every sentence in $\mathcal{S}$, we can formulate a sentence $\Lambda_\mathcal{S}$ which (truly) says of itself that $\mathbb{T}_\mathcal{S}$ does not assess it as true. The sentence $\Lambda_\mathcal{S}$ necessarily does not belong to $\mathcal{S}$, but $\mathcal{S}$ can be enlarged to include it and $\mathbb{T}_\mathcal{S}$ extended to apply to it, thereby generating a new, larger collection of sentences equipped with a more extensive partial truth predicate. This is related to the notorious ``revenge problem'' which has plagued numerous attempts to resolve the liar paradox, and it reveals an indefinitely extensible aspect of classical truth, to use Dummett's term (\cite{dummett2}, p.\ 441). One of Dummett's signal contributions was his thesis that intuitionistic logic is the appropriate form of logic to be used when reasoning about indefinitely extensible concepts.

Meaningfulness is relevant here because sentences that reference their own truth do not always have a clear meaning. Thus once a self-applicative truth predicate is in play, we cannot assume that every syntactically correct sentence is meaningful, and if we do not know what $\phi$ means then we cannot use Tarski's biconditional $\phi \leftrightarrow \mathbb{T}[\phi]$ to say what it means for $\phi$ to be true.

The relevance of assertibility is that we can use it to reason about sentences that might not be meaningful. We can reason {\it subjunctively} via statements of the form ``if $\phi$ were meaningful, then we could assert $\psi$''. In particular, if $\phi$ ranges over a family of sentences which are not all known to be meaningful, then, for the reason just mentioned, we are not in a position to affirm all corresponding instances of the T-scheme. What we can affirm is all instances of the scheme $$\mathbb{M}[\phi] \to \mathbb{A}[\phi \leftrightarrow \mathbb{T}[\phi]]$$ (if $\phi$ is meaningful, then $\ulcorner \phi \leftrightarrow$ ``$\phi$'' is true$\urcorner$ is assertible).

In Section 1 I make some brief philosophical comments about assertibility. A formal system for reasoning about assertibility, truth, and meaningfulness is presented in Section 2, some theorems provable within it are given in Section 3, and its consistency is proven in Section 4.

\section{}

``Assertibility'', ``truth'', and ``meaningfulness'' are all philosophically loaded terms, but as far as the latter two are concerned, I think the axioms in the next section should be relatively uncontroversial. The principal feature I assume of meaningfulness is that it should be {\it compositional}, i.e., a compound formula is meaningful if and only if its constituent subformulas are meaningful (axiom scheme (2) below). The two properties of truth I require are, first, that $\phi$ and $\mathbb{T}[\phi]$ should be equivalent when $\phi$ is meaningful (i.e., the T-scheme), and second, that meaningless statements cannot be true. These appear as schemes (7) and (8) below.

Assertibility --- constructive truth --- is a subtle notion that is described in somewhat different ways by different authors. The formulation I prefer is: a sentence is assertible if we can be rationally compelled to affirm it. Schemes (4) and (5) below simply say that if $\phi$ and $\psi$ are both assertible, then so is $\phi \wedge \psi$, and if $\phi$ and $\phi \to \psi$ are both assertible then so is $\psi$. These should be more or less naively self-evident.

The axiomatization of assertibility used in \cite{weaver} also includes the capture law $\phi \to \mathbb{A}[\phi]$ for any meaningful sentence $\phi$ (scheme (6) below). This may be seen as a version of Dummett's {\it principle K}, ``if a statement is true, it must be in principle possible to know that it is true'' \cite{dummett}, and it might be defensible on the sort of verificationist grounds he champions. It could also be seen as a consequence of a constructivist picture of a mathematical universe that is made increasingly explicit by the adoption of successively stronger axioms.

Many authors have taken the converse release law $\mathbb{A}[\phi] \to \phi$ to be self-evident: if we are rationally compelled to affirm $\phi$, then surely $\phi$ must actually be the case. However, our naive conviction that there can be no rationally compelling proof of a falsehood becomes problematically circular when it is used to justify the release law and that law is subsequently used in arguments that are then recognized as rationally compelling. It is something like having an axiom system that affirms its own consistency. Therefore we do not include the release law in our axiomatization of assertibility. This point is discussed further in Section 4.4 of \cite{weaver}.

On the other hand, whenever we prove a sentence of the form $\mathbb{A}[\phi]$ constructively, the proof is supposed to actually provide us with a proof of $\phi$, and we ought to then be able to recognize this and infer $\phi$ itself. So there should at least be some settings in which we have a deduction rule of the form ``given $\mathbb{A}[\phi]$, infer $\phi$''. I hasten to add that this rule, taken as a global principle, is subject to the same criticism as the implication $\mathbb{A}[\phi] \to \phi$. (It assumes that every rationally compelling proof of $\mathbb{A}[\phi]$ does succeed in providing a legitimate proof of $\phi$, which could become circularly self-affirming if the release rule is then accepted as a global principle that can be used in rationally compelling proofs.) Still, it should be justifiable in some cases on the grounds just mentioned. I will use this rule below, but as I will explain, only for cosmetic purposes; it is not essential here.

(To be clear, this could be a ``deduction rule'' in the context of Hilbert-style deduction. In natural deduction it could only be used when no temporary premises are in play.)

The implication form of release is paradoxical in relation to the sentence ``this sentence is not assertible''. Let $\Lambda_a$ be this sentence, so that $\Lambda_a$ is, by definition, equivalent to the sentence $\neg \mathbb{A}(\Lambda_a)$. If we adopt the implication form of release then a contradiction can be derived in the following way.
Combine the $\mathbb{A}(\Lambda_a) \to \Lambda_a$ instance of release with the definitional equivalence $\Lambda_a \leftrightarrow \neg \mathbb{A}(\Lambda_a)$; this yields $\mathbb{A}(\Lambda_a) \to \neg\mathbb{A}(\Lambda_a)$, from which we can infer $\neg\mathbb{A}(\Lambda_a)$, and from this simply $\Lambda_a$. Thus we have proven $\Lambda_a$, which shows that $\Lambda_a$ is assertible, contradicting our previous conclusion that it is not assertible.

Excluded middle is not needed for this argument, and the capture law appears only in the incontestable deduction rule form where we infer $\mathbb{A}(\Lambda_a)$ after actually having proven $\Lambda_a$. The implication form of release is the bad actor here.

If we allow only the weaker deduction rule form of release, then $\Lambda_a$ is no longer paradoxical, though it comes very close. What we can prove then is (1) if $\Lambda_a$ is assertible then everything is assertible, and (2) $\Lambda_a$ is not not assertible. See Theorem \ref{lp} below. I will say that $\mathbb{A}(\Lambda_a)$ is {\it anomalous}: it is not not the case, but if it is the case then $0 = 1$ is assertible. This is not a happy conclusion, but there is no actual contradiction here.

\section{}

I will now present a minimalistic formal system ATM which treats sentences that can talk about each other's, or their own, assertibility, truth, and meaningfulness, but nothing else. It is a more elaborate version of the system for pure assertibility discussed in Section 5.1 of \cite{weaver}. The goal is simply to show that self-applicative predicates for assertibility, truth, and meaningfulness can be reasoned about consistently, that using intuitionistic logic and weakening release from an implication to a deduction rule together suffice to neutralize all of the obvious elementary paradoxes.

The language of ATM is specified as follows. First, we introduce the constant symbol $\dot{\bot}$ and a sequence of constant symbols $L_1, L_2, \ldots$. Each constant symbol is a {\it term}, and whenever $s$ and $t$ are terms so are $(s\, \dot{\wedge}\, t)$, $(s\, \dot{\vee}\, t)$, $(s\, \dot{\to}\, t)$, $\dot{\mathbb{A}}[t]$, $\dot{\mathbb{T}}[t]$, and $\dot{\mathbb{M}}[t]$. I will sometimes omit parentheses, and I will also write ``$\dot{\neg}t$'' as an informal abbreviation of ``$t\, \dot{\to}\, \dot{\bot}$'' and ``$s\,\dot{\leftrightarrow}\, t$'' as an informal abbreviation of ``$(s\,\dot{\to}\, t) \dot{\wedge} (t \,\dot{\to}\, s)$''. All terms are built up from the constant symbols in the above manner; there are no variables in the language.

The atomic sentences of the language are $\bot$, which is interpreted as ``$0 = 1$'', and $\mathbb{A}[t]$, $\mathbb{T}[t]$, and $\mathbb{M}[t]$ for any term $t$. Whenever $\phi$ and $\psi$ are sentences so are $(\phi \wedge \psi)$, $(\phi \vee \psi)$, and $(\phi \to \psi)$. As with terms, I will write ``$\neg\phi$'' as an abbreviation of ``$\phi \to \bot$'' and ``$\phi \leftrightarrow \psi$'' as an abbreviation of ``$(\phi \to \psi) \wedge (\psi \to \phi)$''. This is a propositional language, with no quantification.

For every sentence $\phi$ there is a corresponding term $\dot{\phi}$ simply obtained by putting dots over every $\bot$, $\wedge$, $\vee$, $\to$, $\mathbb{A}$, $\mathbb{T}$, and $\mathbb{M}$ symbol in $\phi$. Conversely, every term $t$ can be ``evaluated'' to a sentence $\hat{t}$ such that $\dot{\phi}$ always evaluates back to $\phi$. This is done by letting $\dot{\bot}$ evaluate to $\bot$; letting $\dot{\mathbb{A}}[t]$, $\dot{\mathbb{T}}[t]$, and $\dot{\mathbb{M}}[t]$ evaluate to $\mathbb{A}[t]$, $\mathbb{T}[t]$, and $\mathbb{M}[t]$, respectively, for any term $t$ (note that $t$ is left unchanged); and letting $s\, \dot{\wedge}\, t$, $s\, \dot{\vee}\, t$, and $s\, \dot{\to}\, t$ evaluate to $\hat{s} \wedge \hat{t}$, $\hat{s} \vee \hat{t}$, and $\hat{s} \to \hat{t}$. We also declare that each $L_i$ evaluates to a specific sentence $\theta_i$. The $\theta_i$ can be chosen in any manner whatever, but for the sake of concreteness let $\theta_1$ and $\theta_2$ be the classical and assertible liar sentences ``$\neg\mathbb{T}[L_1]$'' and ``$\neg\mathbb{A}[L_2]$'', respectively. Thus $\theta_1$ states that $\theta_1$ is not true (since $\hat{L}_1 = \theta_1$), and $\theta_2$ states that $\theta_2$ is not assertible (since $\hat{L}_2 = \theta_2$).

We could also formulate in this language a ``liar pair'' of sentences that affirm each other's truth and untruth, a variant assertible liar that affirms the assertibility of its negation, a variant classical liar that says of itself that if it is  meaningful then it is not true, etc. These sentences all generate the sort of ``obvious elementary paradoxes'' I was referring to earlier.

The logical axioms of ATM consist of all sentences of the form

\begin{quote}
  $(\mathbb{M}[s] \wedge \mathbb{M}[t] \wedge \mathbb{M}[u]) \to \mathbb{A}[\dot{A}(s, t, u)]$,
\end{quote}
for any terms $s$, $t$, and $u$, where $A$ is one of the standard Hilbert-style axiom schemes for intuitionistic propositional logic. (So we get things like $s \,\dot{\wedge}\, t\, \dot{\to}\, s$, $s \,\dot{\wedge}\, t \,\dot{\to}\, t$, $s \,\dot{\to}\, (t \,\dot{\to}\, s \,\dot{\wedge}
\,t)$, etc.)

We need to include a meaningfulness premise in these axioms because ATM involves sentences that reference truth in a potentially circular fashion, and we do not assume that all such sentences are necessarily meaningful. If we accept that meaningfulness is compositional, then sentences like $\phi \wedge \psi \to \psi$ are only meaningful if $\phi$ and $\psi$ are meaningful, and hence we cannot adopt them as axioms for all $\phi$ and $\psi$ in the language. Nor can we say, e.g., ``$(\mathbb{M}[\dot{\phi}] \wedge \mathbb{M}[\dot{\psi}]) \to (\phi \wedge \psi \to \phi)$'', because if $\phi$ or $\psi$ is meaningless this still runs afoul of compositionality. We have to put the axiom in subjunctive form and say ``$(\mathbb{M}[s] \wedge \mathbb{M}[t]) \to \mathbb{A}[s\, \dot{\wedge}\, t\,\dot{\to}\, s]$''.

However, we will adopt nonlogical axioms which state that every sentence that contains truth in a ``grounded'' way is meaningful. Given these axioms, we will be able to use the deduction rules given below to infer the Hilbert axioms for grounded sentences. Thus we reason as usual about grounded sentences, but only subjunctively about ungrounded sentences. Nonetheless, we will still be able to prove interesting things about ungrounded sentences.

{\it Groundedness} of sentences is defined according to the following rules.

\begin{itemize}
\item $\bot$ is grounded

\item $\mathbb{A}[t]$ and $\mathbb{M}[t]$ are grounded, for any term $t$

\item if $\phi$ and $\psi$ are grounded then so are $\phi \wedge \psi$, $\phi \vee \psi$, and $\phi \to \psi$

\item if $\hat{t}$ is grounded then $\mathbb{T}[t]$ is grounded.
\end{itemize}
The grounded sentences constitute the smallest set of sentences satisfying these conditions. This is similar to, but not quite the same as, Kripke's notion of groundedness \cite{kripke}. The idea is that sentences are ungrounded if they reference truth in a circular way.

The nonlogical axioms of ATM consist of the following schemes (1) through (9). They are all schemes because $s$ and $t$ could be any terms.

\begin{enumerate}
\item $\mathbb{M}[t]$, for any term $t$ such that $\hat{t}$ is grounded

\item $(\mathbb{M}[s] \wedge \mathbb{M}[t]) \leftrightarrow  \mathbb{M}[s \, \dot{\wedge}\, t] \leftrightarrow \mathbb{M}[s \, \dot{\vee}\, t]  \leftrightarrow \mathbb{M}[s\, \dot{\to}\, t]$

\item $\mathbb{A}[t] \to \mathbb{M}[t]$

\item $(\mathbb{A}[s] \wedge \mathbb{A}[t]) \to  \mathbb{A}[s\, \dot{\wedge}\, t]$

\item $(\mathbb{A}[s] \wedge \mathbb{A}[s \, \dot{\to}\, t]) \to \mathbb{A}[t]$

\item $\mathbb{M}[t] \to \mathbb{A}[t \dot{\to} \dot{\mathbb{A}}[t]]$

\item $\mathbb{M}[t] \to \mathbb{A}[t\, \dot{\leftrightarrow}\, \dot{\mathbb{T}}[t]]$

\item $\neg \mathbb{M}[t] \to \mathbb{A}[\dot{\neg}\dot{\mathbb{T}}[t]]$

\item $\mathbb{M}[t] \to \mathbb{A}[t\, \dot{\leftrightarrow}\, t']$, for any terms $t$ and $t'$ with $\hat{t} = \hat{t}'$.
\end{enumerate}

Scheme (1) affirms the meaningfulness of all grounded sentences; scheme (2) says that a complex sentence is meaningful if and only if its components are; scheme (3) affirms that anything assertible is meaningful. Schemes (4) and (5) describe two basic inferences involving assertibility, and (6) is the capture scheme, stated subjunctively to allow for the possibility that $\hat{t}$ is meaningless. Schemes (7) and (8) concern what one can say about $\mathbb{T}[t]$ if $\hat{t}$ is definitely either meaningful or not meaningful, and scheme (9) expresses the idea that any two terms which evaluate to the same sentence are interchangeable.

ATM has three rules of inference:

\begin{enumerate}
\item given $\phi$ and $\psi$, infer $\phi \wedge \psi$ (conjunction)

\item given $\phi$ and $\phi \to \psi$, infer $\psi$ (modus ponens)

\item given $\mathbb{A}[t]$, infer $\hat{t}$ (release).
\end{enumerate}

As I mentioned earlier, we will only use the release rule for cosmetic purposes. Namely, we will use it to get deduction rule versions of the logical axiom schemes and the nonlogical schemes (6) through (9) (e.g.: given $\mathbb{M}[t]$, infer $\hat{t} \to \mathbb{A}[t]$, etc.). So alternatively we could omit the release rule and include deduction rule versions of the logical schemes and the nonlogical schemes (6) through (9).

I also mentioned earlier that we can reason subjunctively about ungrounded sentences. Specifically, working with sentences of the form $(\mathbb{M}[t_1] \wedge \cdots \wedge \mathbb{M}[t_n]) \to \mathbb{A}[\dot{A}(t_1, \ldots, t_n)]$ and using scheme (5), we can prove $(\mathbb{M}[t_1] \wedge \cdots \wedge \mathbb{M}[t_n]) \to \mathbb{A}[\dot{B}(t_1, \ldots, t_n)]$ for any theorem $B(\hat{t}_1, \ldots, \hat{t}_n)$ of the intuitionistic propositional calculus.

\section{}

With regard to the specific sentences $\theta_1$ and $\theta_2$ defined above (the classical and assertible liar sentences), we can say the following.

\begin{theorem}\label{lp}
  ATM proves

  (a) $\mathbb{M}[L_1] \to \mathbb{A}[\dot{\bot}]$

  (b) $\neg\neg\mathbb{M}[L_1]$

  (c) $\mathbb{A}[L_2] \to \mathbb{A}[\dot{\bot}]$

  (d) $\neg\neg\mathbb{A}[L_2]$.
\end{theorem}

\begin{proof}
All the sentences appearing in this proof are grounded and therefore meaningful (though some terms appearing in these sentences, e.g., the ``$L_1$'' in $\neg\neg\mathbb{M}[L_1]$, evaluate to ungrounded sentences). This means that the usual Hilbert-style axioms for intuitionistic logic can be used. For if $\hat{s}$, $\hat{t}$, and $\hat{u}$ are grounded then we get $\mathbb{M}[s]$, $\mathbb{M}[t]$, and $\mathbb{M}[u]$ from scheme (1), then $\mathbb{M}[s] \wedge \mathbb{M}[t] \wedge \mathbb{M}[u]$ from the conjunction rule, and from this we can infer $\mathbb{A}[\dot{A}(s,t,u)]$ from any of the logical axioms by modus ponens. The Hilbert-style axiom $A(\hat{s}, \hat{t}, \hat{u})$ then follows using the release rule. So we can reason normally in intuitionistic logic when all the sentences in play are grounded.

To prove part (a), assume $\mathbb{M}[L_1]$. Infer $\mathbb{A}[L_1\, \dot{\leftrightarrow}\, \dot{\mathbb{T}}[L_1]]$ from (7), then infer $\mathbb{M}[L_1\, \dot{\leftrightarrow}\, \dot{\mathbb{T}}[L_1]]$ from (3) and use (2) to get first $\mathbb{M}[\dot{\mathbb{T}}[L_1]]$ and then $\mathbb{M}[\dot{\neg}\dot{\mathbb{T}}[L_1]]$, i.e., $\mathbb{M}[\dot{\theta}_1]$. Now use (9) with $t = L_1$ and $t' = \dot{\theta}_1$ and reason under $\mathbb{A}$ to get $\mathbb{A}[\dot{\theta}_1\, \dot{\leftrightarrow}\, \dot{\mathbb{T}}[L_1]]$, i.e., $\mathbb{A}[\dot{\neg}\dot{\mathbb{T}}[L_1] \, \dot{\leftrightarrow}\, \dot{\mathbb{T}}[L_1]]$, and from this $\mathbb{A}[\dot{\bot}]$. We have shown that $\mathbb{M}[L_1] \to \mathbb{A}[\dot{\bot}]$.

For part (b), assume $\neg\mathbb{M}[L_1]$. Use (8) to get $\mathbb{A}[\dot{\neg}\dot{\mathbb{T}}[L_1]]$, i.e., $\mathbb{A}[\dot{\theta}_1]$, then (3) to get $\mathbb{M}[\dot{\theta}_1]$. This yields $\mathbb{A}[\dot{\theta}_1 \dot{\leftrightarrow} L_1]$ by (9), then $\mathbb{M}[\dot{\theta}_1 \dot{\leftrightarrow} L_1]$ by (3), and then $\mathbb{M}[L_1]$ by (2). Combining this with the initial $\neg\mathbb{M}[L_1]$ yields $\bot$ (modus ponens), so we have proven $\neg\mathbb{M}[L_1] \to \bot$, i.e., $\neg\neg\mathbb{M}[L_1]$.

For part (c), assume $\mathbb{A}[L_2]$. Infer $\mathbb{M}[L_2]$ from (3), then $\mathbb{A}[L_2 \leftrightarrow \dot{\theta}_2]$ from (9), and then use schemes (3) and (2) to infer $\mathbb{M}[\dot{\theta}_2]$ as well. We are now in a position to assume meaningfulness and reason under $\mathbb{A}$ to get $\mathbb{A}[L_2 \to \dot{\theta}_2]$. (Since $(\phi \leftrightarrow \psi) \to (\phi \to \psi)$ is a theorem of the intuitionistic propositional calculus, and using (5).) Combining this with the assumption of $\mathbb{A}[L_2]$, again using (5), then yields $\mathbb{A}[\dot{\theta}_2]$, i.e., $\mathbb{A}[\dot{\neg}\dot{\mathbb{A}}[L_2]]$. Finally, (6) with $t = L_2$ yields $\mathbb{A}[L_2\, \dot{\to}\, \dot{\mathbb{A}}[L_2]]$, so we get $\mathbb{A}[\dot{\mathbb{A}}[L_2]]$ from (5). Combining that with $\mathbb{A}[\dot{\neg}\dot{\mathbb{A}}[L_2]]$, i.e., $\mathbb{A}[\dot{\mathbb{A}}[L_2] \dot{\to} \dot{\bot}]$, using (5) yet again, finally yields $\mathbb{A}[\dot{\bot}]$. Assuming $\mathbb{A}[L_2]$, we proved $\mathbb{A}[\dot{\bot}]$; thus, we have shown that $\mathbb{A}[L_2] \to \mathbb{A}[\dot{\bot}]$.

For part (d), start by noting that $\neg\mathbb{A}[L_2]$ is grounded and therefore meaningful, so (6) plus the release rule yields $\neg\mathbb{A}[L_2] \to \mathbb{A}[\dot{\neg}\dot{\mathbb{A}}[L_2]]$. Now, assuming $\neg\mathbb{A}[L_2]$, deduce $\mathbb{A}[\dot{\neg}\dot{\mathbb{A}}[L_2]]$, i.e., $\mathbb{A}[\dot{\theta}_2]$, and then use (9) to get $\mathbb{A}[L_2]$, in the same way that we went from $\mathbb{A}[L_2]$ to $\mathbb{A}[\dot{\theta}_2]$ in the proof of part (c). Combining this with the initial $\neg\mathbb{A}[L_2]$ yields $\bot$ (modus ponens), so we have proven $\neg\mathbb{A}[L_2] \to \bot$, i.e., $\neg\neg\mathbb{A}[L_2]$.
\end{proof}

In short, to use the term I introduced earlier, $\mathbb{M}[L_1]$ and $\mathbb{A}[L_2]$ are both anomalous sentences. This suggests a way of thinking about the classical liar paradox. The question ``Is $L_1$ true or not?'' cannot be sensibly posed until we have said what it means for a sentence to be true. If $L_1$ is meaningful then the T-scheme tells us what it means for $L_1$ to be true, and if $L_1$ is not meaningful then we simply define it to be untrue --- but this is a definition by cases, so we have to decide which category to place the liar sentence in before we can make sense of the question of its truth. It is only cogent to ask whether the liar sentence is true if the liar sentence is definitely meaningful or not meaningful, as otherwise we do not know what the assertion that it is true means.

There is no contradiction here. We avoid contradiction by remaining in a state of uncertainty about the meaningfulness of the liar sentence.

Indeed, we can infer that if the classical liar sentence is definitely either meaningful or not meaningful then ``$0 = 1$'' is assertible, and draw the same conclusion if the assertible liar sentence is definitely either assertible or not assertible. Both of these statements follow from Theorem \ref{lp}, using capture to get $\neg\mathbb{M}[L_1] \to \mathbb{A}[\dot{\bot}]$ and $\neg\mathbb{A}[L_2] \to \mathbb{A}[\dot{\bot}]$ from parts (b) and (d).

Another inference we can make is that $\neg \mathbb{A}[\dot{\bot}]$ entails a contradiction --- because assuming $\neg\mathbb{A}[\dot{\bot}]$ converts parts (a) and (c) of Theorem \ref{lp} to $\neg\mathbb{M}[L_1]$ and $\neg\mathbb{A}[L_2]$, contradicting parts (b) and (d). So $\bot$ is not not assertible. This will sound absurd if one is thinking classically, because it would mean that $\bot$ is assertible, but not if one is thinking constructively. Constructively, $\neg\neg\mathbb{A}[\dot{\bot}]$ could be glossed as ``we cannot rule out the possibility that $\bot$ is assertible'', which is exactly why we do not adopt the implication form of release.

\section{}

\begin{theorem}\label{lpcons}
  ATM is consistent.
\end{theorem}

\begin{proof}
  We will construct a set $\mathcal{F}$ of sentences that contains $\bot$, does not contain any of the axioms, and
  whose complement is stable under the deduction rules. This shows that all the theorems of ATM lie
  in the complement of $\mathcal{F}$, and therefore $\bot$ cannot be a theorem.

  The desired set is constructed in an infinite series of {\it levels}, each of
  which consists of an infinite series of {\it stages}, each of which consists of an infinite series of {\it steps}.
  At each step new sentences are added to $\mathcal{F}$; nothing is ever removed.

  Let $S$ be any set of implications. We will use the following rules:

  \begin{itemize}
  \item ($\wedge$ rule) if $\phi \in \mathcal{F}$ and $\psi$ is any sentence, place $\phi \wedge \psi$
    and $\psi \wedge \phi$ in $\mathcal{F}$

  \item ($\vee$ rule) if $\phi, \psi \in \mathcal{F}$, place $\phi \vee \psi$ in $\mathcal{F}$

  \item ($\mathbb{T}$ rule) if $\phi \in \mathcal{F}$ then place $\mathbb{T}[t]$ in $\mathcal{F}$ for every
    term $t$ that evaluates to $\phi$.

  \item ($\to$ rule) for any implication $\phi \to \psi$ in $S$ such that  $\phi \in \mathcal{F}^c$ and
    $\psi \in \mathcal{F}$, place $\phi \to \psi$ in $\mathcal{F}$.
  \end{itemize}

  Define a function $i$ (``implication complexity'') from the set of sentences to $\omega$ by letting $i(\phi)$ be the
  number of appearances of the symbol ``$\to$'' in $\phi$, and define a function $n$ (``nesting'') from the set of
  sentences to $\omega + 1$ by setting $n(\phi) = 0$ when $\phi$ is $\bot$, $\mathbb{M}[t]$, or $\mathbb{A}[t]$, for any
  term $t$; $n(\phi \wedge \psi) = n(\phi \vee \psi) = n(\phi \to \psi) = {\rm max}(n(\phi), n(\psi))$;
  $n(\mathbb{T}[t]) = n(\hat{t}) + 1$; and $n(\phi) = \omega$ for any ungrounded sentence $\phi$.
  These functions will be used to construct $\mathcal{F}$.

  When we are at a given step of the construction I will write $\mathcal{F}_-$ for the state of $\mathcal{F}$ going
  into that step.

  The construction goes through $\omega$ levels labelled $p = 0$, $1$, $\ldots$ (``principal steps''). Every level has
  $\omega + 1$ stages labelled $n = 0$, $1$, $\ldots$, $\omega$ (corresponding to nesting). Stage $0$ on any level has $\omega$
  steps labelled $i = 0$, $1$, $\ldots$ and every other stage has $\omega$ steps labelled $i = 1$, $2$, $\ldots$ (corresponding
to implication
  complexity).

  At step 0, stage 0, level 0, when $\mathcal{F}_- = \emptyset$, we put $\bot$ in $\mathcal{F}$ and then close up
  under the $\wedge$, $\vee$, and $\mathbb{T}$ rules. That is, at the end of this step $\mathcal{F}$ is the smallest
  set that contains $\bot$ and is stable under the $\wedge$, $\vee$, and $\mathbb{T}$ rules.

  At step 0, stage 0, level 1, we place $\mathbb{M}[t]$ and $\mathbb{A}[t]$ in $\mathcal{F}$ for every term $t$
  whose evaluation $\hat{t}$ is ungrounded, and also every $\mathbb{A}[t]$ such that $\hat{t} \in \mathcal{F}_-$.
  Then we close up under the $\wedge$, $\vee$, and $\mathbb{T}$ rules.

  At step 0, stage 0 of any other level, we place every $\mathbb{A}[t]$ such that $\hat{t} \in \mathcal{F}_-$
  in $\mathcal{F}$, then close up under the $\wedge$, $\vee$, and $\mathbb{T}$ rules.

  At step $i \geq 1$, stage $n$ of any level (including $n = \omega$), let $S$ be the set of implications $\phi \to \psi$ with
  $n(\phi \to \psi) = n$ and $i(\phi \to \psi) = i$, and apply the $\to$ rule to $\mathcal{F}_-$, then
  close up under the $\wedge$, $\vee$, and $\mathbb{T}$ rules. (We do not attempt to close up under the $\to$ rule.)

  That completes the description of $\mathcal{F}$. Now we have to verify that its complement, once the construction is complete, contains the axioms of ATM and is stable under its deduction rules. Three helpful observations are

  \begin{itemize}
  \item if the sentence $\theta$ is added to $\mathcal{F}$ at step $i$ of stage $n$ on any level, then either
    $n(\theta) > n$ or else $n(\theta) = n$ and $i(\theta) \geq i$

  \item if $\phi \to \psi$ is added in to $\mathcal{F}$ anywhere on a given level, then at the end of that level
    $\phi$ will not (yet) have been added to $\mathcal{F}$

  \item at step 0, stage 0 of any level, if $\phi \not\in \mathcal{F}_-$ and $\psi \in \mathcal{F}_-$ then
    $\phi \to \psi \in \mathcal{F}_-$ (i.e., at these points in the construction $\mathcal{F}_-$ is stable under
    the $\to$ rule with $S =$ all implications).
  \end{itemize}
  The first observation is trivial when $i = n = 0$, and at any step $i \geq 1$, if $\phi \to \psi$ is added in
  to $\mathcal{F}$ then $n(\phi \to \psi) = n$ and $i(\phi \to \psi) = i$. Then we close up under the $\wedge$, $\vee$, and
  $\mathbb{T}$ rules, but the set of sentences $\theta$ with $n(\theta) > n$ or $n(\theta) = n$, $i(\theta) \geq i$ is
  stable under these rules. The second observation can be seen by noting that when $\phi \to \psi$ is added in at
  stage $n$ and step $i$, at that point $\phi$ does not belong to $\mathcal{F}$, and we also have either
  $n(\phi) < n(\phi \to \psi) = n$ or else $n(\phi) = n(\phi \to \psi) = n$ and $i(\phi) < i(\phi \to \psi) = i$.
  So by the first observation, $\phi$ has already missed any opportunity to be
  added in during the current level. The third observation holds because $n(\psi) \leq n(\phi \to \psi)$
  and $i(\psi) < i(\phi \to \psi)$, so if $\psi$ appears at any point in some level and $\phi$ does not,
  the step where $\phi \to \psi$ is
  potentially added in must come later in that level, and $\phi \to \psi$ {\it will} get added in at this later step.

  We verify that $\mathcal{F}^c$ is stable under the three deduction rules. First, the only way $\phi \wedge \psi$
  could be added is at one of the many points where we close up under the $\wedge$ rule. This means that
  if $\phi \wedge \psi$ appears in $\mathcal{F}$, then either $\phi$ or $\psi$ must also appear, showing
  that the complement of $\mathcal{F}$ is stable under the conjunction deduction rule: if neither
  $\phi$ nor $\psi$ lies in $\mathcal{F}$ then $\phi \wedge \psi$ cannot either.

  For modus ponens, suppose $\psi$ belongs to $\mathcal{F}$ and let $\phi$ be arbitrary. We want to show that
  either $\phi$ or $\phi \to \psi$ will also belong to $\mathcal{F}$. This follows from the ``second observation''
  made above: at whatever level $\psi$ first appears on, if $\phi$ has not appeared by the end of that level then
  $\phi \to \psi$ has. So if both $\phi$ and $\phi \to \psi$ are in $\mathcal{F}^c$, then so is $\psi$.

  The release rule holds because if $\psi$ appears in $\mathcal{F}$ at any level, then $\mathbb{A}[t]$ will
  appear in step 0, stage 0 of the following level, for any term $t$ that evaluates to $\psi$. Thus $\mathbb{A}[t] \in
  \mathcal{F}^c$ implies $\hat{t} \in \mathcal{F}^c$.

  Next, let us show that none of the nonlogical axioms belongs to $\mathcal{F}$. The first comment is that
  no instance of scheme (1) can appear at any point because the only place sentences of the form
  $\mathbb{M}[t]$ are placed in $\mathcal{F}$ is step 0, stage 0, level 1, and this is only done if
  $\hat{t}$ is ungrounded.
  Nor can any instance of scheme (2) be placed in $\mathcal{F}$ at any point, because if either $\hat{s}$ or
  $\hat{t}$ is ungrounded then all of $\mathbb{M}[s] \wedge \mathbb{M}[t]$, $\mathbb{M}[s \dot{\wedge} t]$,
  $\mathbb{M}[s \dot{\vee} t]$, and $\mathbb{M}[s \dot{\to} t]$ are placed $\mathcal{F}$ simultaneously, at
  step 0, stage 0, level 1; and if $\hat{s}$ and $\hat{t}$ are both grounded, none of them is ever placed in
  $\mathcal{F}$.

  Closing up under the $\wedge$, $\vee$, and $\mathbb{T}$ rules never introduces any of the axioms from schemes
  (3) through (9), because these are all implications. Nor are any of these axioms introduced at step 0 of any
  stage and level, because no implications are introduced in those steps. Thus
  we must check that no instance of any of the schemes (3) --- (9) can be added to $\mathcal{F}$ using
  the $\to$ rule at some step $i \geq 1$. Taking scheme (3), we see that the $\to$ rule will never add in
  something of the form $\mathbb{A}[t] \to \mathbb{M}[t]$, because if $\hat{t}$ is ungrounded then
  $\mathbb{A}[t]$ and $\mathbb{M}[t]$ both appear in $\mathcal{F}$ at the same time,
  and if $\hat{t}$ is grounded then $\mathbb{M}[t]$ never appears. So there is never a point
  at which $\mathbb{M}[t]$ has appeared and $\mathbb{A}[t]$ has not, which is a requirement to
  add $\mathbb{A}[t] \to \mathbb{M}[t]$ with the $\to$ rule.
  For scheme (4) observe that if either $\hat{s}$ or $\hat{t}$ is ungrounded then both sides of the implication
  appear in $\mathcal{F}$ at the same time, at step 0 of stage 0, level 1. If both are grounded then
  $\mathbb{A}[s \dot{\wedge} t]$ can only be added on step 0, stage 0 of some level such that $\hat{s}\wedge\hat{t}$
  first appeared in $\mathcal{F}$ somewhere on the previous level. But
  $\hat{s}\wedge\hat{t}$ can never be added before both $\hat{s}$ and $\hat{t}$, so at least one of
  $\mathbb{A}[s]$ and $\mathbb{A}[t]$ must be added at the same time $\mathbb{A}[s\dot{\wedge} t]$ is added,
  and thus $\mathbb{A}[s] \wedge \mathbb{A}[t]$ must also appear at the same step.

  For scheme (5), the ``second observation'' made earlier ensures that either $\hat{s}$ or $\hat{s} \to \hat{t}$
  will appear at the same level as $\hat{t}$, or earlier, so at least one of $\mathbb{A}[s]$ and
  $\mathbb{A}[s\dot{\to} t]$ (and in either case, their conjunction)
  would appear at the same step as $\mathbb{A}[t]$, at the latest.

  The left side of scheme (8) is placed in $\mathcal{F}$ at step 1, stage 0, level 0 (regardless of whether $\hat{t}$
  is grounded), whereas the right side could not be placed in $\mathcal{F}$ at any point in level 0.

  In schemes (6), (7), and (9), if $\hat{t}$ is ungrounded then both sides are placed in $\mathcal{F}$ at the same
  time, at step 0, stage 0, level 1. So assume $\hat{t}$ is grounded. Every instance of these three schemes is an
  implication with a conclusion of the form $\mathbb{A}[u]$.
  Moreover, groundedness of $\hat{t}$ entails that $\hat{u}$ is grounded in
  every case. So $\mathbb{A}[u]$ could only appear subsequent to $\hat{u}$ appearing, and thus it will suffice
  to show that $\hat{u}$ never appears. In scheme (6) we get this because groundedness of $\hat{t}$ entails that
  $\mathbb{A}[t]$ cannot appear before $\hat{t}$; in scheme (7) we never get either $\hat{t} \to \mathbb{T}[t]$
  or $\mathbb{T}[t] \to \hat{t}$, regardless of whether $\hat{t}$ is grounded, because the $\mathbb{T}$ rule is
  always invoked just after $\hat{t}$ is included, before we get a chance to apply the
  $\to$ rule to them. And scheme (9) is easy because $\mathbb{A}[t\, \dot{\leftrightarrow}\, t']$ can only be
  added after $\hat{t} \leftrightarrow \hat{t}'$ was added, and the latter can never happen, since $\hat{t} =
  \hat{t}'$.

  Finally, we must show that none of the logical axioms can appear in $\mathcal{F}$. As these are all implications,
  we only have to show that none of them can be added at any step using the $\to$ rule. This clearly cannot happen
  if any of $s$, $t$, or $u$ evaluates to an ungrounded sentence, because in that case the premise of the implication
  would have appeared at step 0, stage 0, level 1, and its conclusion could appear at that point at the earliest. If
  they all evaluate to grounded sentences, then we show that $\mathbb{A}[\dot{A}(s,t,u)]$ can never appear by checking
  that $A(\hat{s}, \hat{t}, \hat{u})$ never appears.

  This can be seen using the three observations made earlier in the proof. The verification is straightforward but tedious, so I will simply illustrate the technique for the most complicated axiom scheme,
  where $A(\hat{s}, \hat{t}, \hat{u})$ is the sentence $(\hat{s} \to (\hat{t} \to \hat{u})) \to ((\hat{s} \to \hat{t})
  \to (\hat{s} \to \hat{u}))$. If a sentence of this form
  appeared in $\mathcal{F}$ at any point,
  the second observation made above tells us that there would have to be a level at which $(\hat{s} \to \hat{t}) \to
  (\hat{s} \to \hat{u})$ is placed in $\mathcal{F}$ but $\hat{s} \to (\hat{t} \to \hat{u})$ is not. Then
  $\hat{s} \to \hat{u}$ must be placed in $\mathcal{F}$ in this level, as if it had appeared at an earlier level
  then $(\hat{s} \to \hat{t}) \to (\hat{s} \to \hat{u})$ would have too. So at the end of
  this level $\hat{u}$ is in $\mathcal{F}$ and $\hat{s}$ is not, and also $\hat{s} \to \hat{t}$ is not,
  so that since $\hat{s}$ is not in $\mathcal{F}$, $\hat{t}$ cannot be either. Thus at the end of this level we
  will have $\hat{u}$ in $\mathcal{F}$, but not $\hat{s}$ or $\hat{t}$. The second observation now tells us first that
  $\hat{t} \to \hat{u}$ must be in $\mathcal{F}$ at this point, and second that $\hat{s} \to (\hat{t} \to \hat{u})$ must
  be in $\mathcal{F}$ at this point. But this contradicts the earlier statement that
  $\hat{s} \to (\hat{t} \to \hat{u})$ does not appear in $\mathcal{F}$ at this level. We conclude that
  there is no level where $(\hat{s} \to (\hat{t} \to \hat{u})) \to ((\hat{s} \to \hat{t})
  \to (\hat{s} \to \hat{u}))$ could be added. The other logical axioms are handled similarly.
\end{proof}

It could also be reasonable to adopt deduction rules which, given $\hat{t}$, let us infer $\mathbb{A}[t]$, $\mathbb{T}[t]$, and $\mathbb{M}[t]$: any theorem of a system we trust should be assertible, true, and meaningful. But in the ATM setting these rules are superfluous. This is a consequence of the next result, which follows easily from the proof of Theorem \ref{lpcons}.

\begin{corollary}\label{lpcor}
  Every theorem of ATM is grounded.
\end{corollary}

\begin{proof}
  Go by induction on the length of a proof in ATM. All of the axioms are obviously grounded, and neither the
  conjunction rule nor modus ponens can derive an ungrounded sentence from grounded ones. The release rule
  could potentially do this, but we saw in the proof of Theorem \ref{lpcons} that $\mathbb{A}[t]$ cannot be
  a theorem if $\hat{t}$ is ungrounded, because every $\mathbb{A}[t]$ with $\hat{t}$ ungrounded is placed in
  $\mathcal{F}$ at step 0, stage 0, level 1. Thus the release rule also cannot introduce any ungrounded sentences.
\end{proof}

Axiom scheme (1) says that every grounded sentence is meaningful, so it follows from Corollary \ref{lpcor} that every theorem of ATM is meaningful, provably in ATM. If $\hat{t}$ is any theorem of ATM, schemes (6) and (7) then yield $\mathbb{A}[t\, \dot{\to}\, \dot{\mathbb{A}}[t]]$ and $\mathbb{A}[t\, \dot{\leftrightarrow}\, \dot{\mathbb{T}}[t]]$, and applying the release rule and modus ponens shows
that $\mathbb{A}[t]$ and $\mathbb{T}[t]$ are also theorems of ATM.

\bibliographystyle{amsplain}

\end{document}